\documentclass[12pt,a4paper]{article}
\usepackage[utf8]{inputenc}
\usepackage{amsthm}
\usepackage{amsmath}
\usepackage{amsfonts}
\usepackage{amssymb}
\usepackage{verbatim}
\usepackage{enumitem}
\usepackage{cite}
\newcommand{\ka}{\mathbf{k}}
\newcommand{\ot}{\otimes}
\newcommand{\ra}{\longrightarrow}
\newcommand{\la}{\lambda}
\newtheorem{theorem}{Theorem}[section]
\newtheorem{lem}[theorem]{Lemma}
\newtheorem{prop}[theorem]{Proposition}
\newtheorem{cor}[theorem]{Corollary}
\theoremstyle{definition}

\usepackage{tikz-cd}
\usepackage[margin=1in]{geometry}
\usepackage{microtype}
\usepackage{indentfirst}
\usepackage[small]{titlesec}
\date{}

\author{Marcin Chałupnik and Patryk Jaśniewski}
\title{On strict polynomial functors with bounded domain}
\begin{document}
\maketitle

\begin{abstract}
We introduce a new functor category: the category $\mathcal{P}_{d,n}$ of strict polynomial functors with bounded by $n$ domain of degree $d$ over a field of characteristic $p>0$. It is equivalent to the category of finite dimensional modules over the Schur algebra $S(n,d)$, hence it allows one to apply the tools available in functor categories to representations of the algebraic group $\operatorname{GL}_n$.  We investigate in detail the homological algebra in ${\cal P}_{d,n}$ for $d=p$ and establish equivalences between certain subcategories of ${\cal P}_{d,n}$'s which resemble the Spanier-Whitehead duality in stable homotopy theory.
\newline

\end{abstract}
\section*{Introduction}
It was demonstrated in a breakthrough work \cite{FLS} that the category
${\cal F}$ of functors from the category of finite dimensional vector spaces over a a finite field {\bf k} to the category of all vector spaces over {\bf k} is a valuable tool for investigating representations of $\operatorname{GL}_n({\bf k})$ for $n>>0$. This approach has turned out to be particularly effective in a study of homological problems for $\operatorname{GL}_n({\bf k})$-modules, like computing Ext-groups between important objects. Even more impressive results were obtained by using the category ${\cal P}_d$ of strict polynomial functors of degree $d$ introduced by Friedlander and Suslin in \cite{fs}, which may be thought of as an algebro-geometric version of ${\cal F}$. This variant allows one to enjoy both flexibility offered by a functor category and various rich algebraic structures like the structure of a highest weight category, block theory etc.  What is interesting here is that  it can be shown that  ${\cal P}_d$ is equivalent to the category of finite dimensional modules over the Schur algebra 
$S(n,d)$ for $n\geq d$ \cite[Thm. 3.2]{fs}, which is a classically studied object. However, some crucial technical tools available in ${\cal P}_d$, like the sum-diagonal adjunction, are completely invisible from the level of modules over the Schur algebra.

In the present article we introduce the functor category ${\cal P}_{d,n}$, which is a modification of ${\cal P}_{d}$ dedicated to the study of representations of $\operatorname{GL}_n({\bf k})$  also for small values of $n$. Technically, our definition is quite straightforward: our definition differs from that of ${\cal P}_d$ by the fact we  restrict the domain of functors to the spaces of dimension at most $n$. Thus our category is closely related to ${\cal P}_d$ and shares its basic properties, however, one quickly realizes that it is much more complicated than ${\cal P}_d$. In particular, we do not have the sum-diagonal adjunction anymore, which has some interesting consequences (see Section 2). Nevertheless, the functorial approach is still beneficial in our situation as we demonstrate in Section 2 by providing some non-trivial Ext-computations. We also introduce in Section 3 a certain  duality, which is closely related to the monoidal duality for $\operatorname{GL}_n$-modules, though it seems not to have any counterpart for ${\cal P}_d$. We think of our work as the starting point for a systematic study of functor categories in ``the unstable case'' and we hope to develop our investigations in future works.

Now we briefly describe the contents of the article.
In Section 1 we introduce the category ${\cal P}_{d,n}$ and establish its basic properties. We show  that it is equivalent to the category of finite dimensional 
modules over $S(n,d)$ (Thm. 1.2) and we recall the known relationship of the latter with the category of representations of the algebraic group $\operatorname{GL}_n$. We also study the relationship between ${\cal P}_{d,n}$ and ${\cal P}_d$, which is best understood in terms of recollement diagrams (Thm. 1.4). It will be an important tool in the computations performed in Section 2.

In Section 2 we study in detail the homological algebra in ${\cal P}_{p,n}$,
where $p$ is the characteristic of ground field. We observe many similarities 
to the case of ${\cal P}_p$, which was studied in \cite{pj} (e.g. the formality of the Yoneda algebra of simples, see Thm. 2.2), but we also notice important differences. In particular, we compute
the groups ${\rm Ext}_{{\cal P}_{p,n}}^*(I^{\otimes p},F)$ for various $F$'s, which are specific to ``the unstable case'' and we relate it to the classical Schur functor (Cor. 2.6). At last we establish quite surprising connection between ${\cal P}_{p,p-1}$ and the category of representations of the symmetric group 
$\Sigma_p$ at the level of $K$-theory (Cor. 2.7).

In Section 3 we establish equivalences between certain subcategories of
${\cal P}_{d,n}$'s, which have formal similarities with the Spanier-Whitehead
duality in stable homotopy theory. Our construction is somewhat surprising, since it provides a functorial interpretation of some classical computations of duals of representations which seemed not to fit into a categorical context.
\section{The category \( \mathcal{P}_{d,n} \)}
We refer to \cite{fs}, \cite{nick} and \cite{pirashvili} for the definition of strict polynomial functors, to \cite{abw} and \cite{weyman} for the definition of Schur and Weyl functors, to \cite{cps_hwc} for elementary properties of highest weight categories, to \cite{krause} for the structure of the category of strict polynomial functors of a given degree as a highest weight category and to \cite{fs} for the definition of the de Rham and Koszul complexes.

In this section we define the category of homogeneous strict polynomial functors of degree $d$ with bounded domain in the spirit of the Pirashvili
 and Kuhn reformulation of the definition of strict polynomial functors 
 (c.f. \cite{nick}, \cite{pirashvili}) and we discuss properties of that category. The category of strict polynomial functors of degree $d$ is denoted by $\mathcal{P}_d$.

Let \( \mathcal{V}ect_{\bf k}^{\leq n} \) be the category of ${\bf k}-$linear spaces with dimension at most \(n.\) We define the category \( \Gamma^d \mathcal{V}ect_{\bf k}^{\leq n} \) as follows. The objects of that category are the same as the objects of \( \mathcal{V}ect_{\bf k}^{\leq n} \) and for two linear spaces $V,W$ we set \( \operatorname{Hom}_{\Gamma^d \mathcal{V}ect_{\bf k}^{\leq n}}(V,W) = \Gamma^d\operatorname{Hom}_{\bf k}(V,W),\) where \( \operatorname{Hom}_{\bf k}(V,W) \) is the linear space of ${\bf k}$-linear maps between $V$ and $W.$ The composition of morphisms is determined by the natural identification \( \Gamma^d\operatorname{Hom}_{\bf k}(V,W) \simeq \operatorname{Hom}_{\bf k}(V^{\otimes d},W^{\otimes d})^{\Sigma_d} \). Finally, the category \( \mathcal{P}_{d,n} \) of homogeneous strict polynomial functors of degree $d$ with bounded domain \( \Gamma^d\mathcal{V}ect_{\bf k}^{\leq n} \) is defined as the category of ${\bf k}$-linear functors from \( \Gamma^d \mathcal{V}ect_{\bf k}^{\leq n} \) to the category of finite dimensional linear spaces \( \mathcal{V}ect_{\bf k} \). \( \mathcal{P}_{d,n} \) is an abelian category. The Kuhn duality $(-)^\#$ in $\mathcal{P}_{d,n}$ is defined as the restriction of the Kuhn duality in $\mathcal{P}_d$ to the category \( \Gamma^d \mathcal{V}ect_{\bf k}^{\leq n} \) and we define the tensor product in $\mathcal{P}_{d,n}$ in the same manner. Let \( \Gamma^{d,n}=\Gamma^d \circ \operatorname{Hom}({\bf k}^n, -)$ and $S^{d,n}= (\Gamma^{d,n})^{\#}.\)
\begin{prop}
\label{proj_gen}
\( \Gamma^{d,n} \) is a compact projective generator of \( \mathcal{P}_{d,n} \) and \( S^{d,n} \) is an injective cogenerator of \( \mathcal{P}_{d,n} \). In particular, \( \mathcal{P}_{d,n} \) has enough injective and projective objects.
\end{prop}
\begin{proof}
By the Yoneda lemma
\begin{equation}
\label{yoneda_projective_eq}
\operatorname{Hom}_{\mathcal{P}_{d,n}}(\Gamma^{d,n}, F) \simeq F({\bf k}^n),
\end{equation}
hence \( \Gamma^{d,n} \) is a projective functor. We also see from (\ref{yoneda_projective_eq}) that \( \Gamma^{d,n} \) is compact. It remains to prove that \( \Gamma^{d,n} \) is a generator of \( \mathcal{P}_{d,n} \), i.e. \( \operatorname{Hom}_{\mathcal{P}_{d,n}}(\Gamma^{d,n},-) \) is a faithful functor. By the Yoneda lemma it is sufficient to show that if \(\eta:F\to G\) is a natural transformation such that $\eta({\bf k}^n)=0$ then $\eta=0$, but it follows from the fact that for any $m<n$, ${\bf k}^m$ is a retract of
${\bf k}^n$ in $\mathcal{V}ect_{\bf k}^{\leq n}$, hence also in 
$\Gamma^d \mathcal{V}ect_{\bf k}^{\leq n}$.

By the Kuhn duality \( S^{d,n} \) is an injective cogenerator of \( \mathcal{P}_{d,n} \). The last statement is evident from the first part of the proposition.
\end{proof}
In the sequel \( \operatorname{GL}_n=\operatorname{GL}_n({\bf k}) \) is regarded as an algebraic group over ${\bf k}$. Let $\operatorname{GL}_n^\text{rat}$-mod be the category of rational left \( \operatorname{GL}_n$-modules. We denote by $\operatorname{GL}^{\text{Pol},d}_n$-mod the category of finite dimensional homogeneous polynomial $\operatorname{GL}_n$-modules of degree $d$. Let $S(n,d)$-mod be the category of finite dimensional left modules over the Schur algebra $S(n,d)$.
\begin{theorem}
\label{schur_equiv_prop}
There are equivalences of abelian categories \( \mathcal{P}_{d,n} \simeq S(n,d) \)-mod and \(\mathcal{P}_{d,n} \simeq \operatorname{GL}^{\text{Pol},d}_n\)-mod . In particular, \( \mathcal{P}_{d,n}\simeq \mathcal{P}_d \) for \( n\geq d \).
\end{theorem}
\begin{proof}
By Proposition 1.1 and the Gabriel theorem (c.f. \cite[II Thm. 1.3]{bass}) we obtain the equivalence \[ \mathcal{P}_{d,n} \simeq \operatorname{End}_{\mathcal{P}_{d,n}}(\Gamma^{d,n})-\text{mod}. \]It follows from (\ref{yoneda_projective_eq}) that \[ \operatorname{End}_{\mathcal{P}_{d,n}}(\Gamma^{d,n})\simeq \Gamma^{d,n}({\bf k}^n)=\Gamma^d (\operatorname{End}({\bf k}^n))\simeq S(n,d), \]hence 
\begin{equation}
\mathcal{P}_{d,n} \simeq \operatorname{End}_{\mathcal{P}_{d,n}}(\Gamma^{d,n})-\text{mod} \simeq S(n,d)-\text{mod}.
\end{equation}
It immediately implies the last statement, because \( \mathcal{P}_d \simeq S(n,d)\)-mod for \(n\geq d\) (c.f. \cite[Thm. 3.2]{fs}).
The second equivalence in Theorem follows from the first one and the well-known equivalence \(S(n,d)-\text{mod}\simeq \operatorname{GL}^{\text{Pol},d}_n-\text{mod}\) (c.f. \cite[Thm. 2.2.7]{martin}). 
\end{proof}
\begin{cor}
\( \mathcal{P}_{d,n} \) is a highest weight category with poset $\Lambda(d,n)$ of Young diagrams of weight $d$ and with at most $n$ columns, with the partial order on $\Lambda(d,n)$ being the reversed dominance order.
\end{cor}
\begin{proof}
The assertion follows immediately from Theorem \ref{schur_equiv_prop} and the well-known fact that \( S(n,d) \)-mod is a highest weight category (c.f. \cite[Thm. 4.1]{parshall}).
\end{proof}

{\bf Remark.} An unfortunate fact that the order is the reversed dominance order  is caused by the fact that we label the family of Schur functors as in 
\cite{abw}, \cite{weyman} and simultaneously follow the general conventions concerning highest weight categories (see e.g \cite{cps_hwc}).

Now we recall the standard definitions and facts about the algebraic group 
$\operatorname{GL}_n$ (see e.g. \cite{jantzen}). Let $T_n\subset \operatorname{GL}_n$ be the standard maximal torus of $\operatorname{GL}_n$, i.e. the subgroup of diagonal matrices. Let $B_n$ be a Borel subgroup of $\operatorname{GL}_n$ consisting of upper triangular matrices. The weight lattice of $\operatorname{GL}_n$, i.e. the group $X(T_n)$ of rational characters of $T_n$, is identified with the group $\mathbb{Z}^n$. Elements of $X(T_n)$ are called weights. For $M\in \operatorname{GL}^{\text{rat}}_n$-mod and $\lambda \in X(T_n)$ the weight space $M_\lambda$ is the linear space $M_\lambda=\{m\in M: tm=\lambda(t)m \text{ for all } t\in T_n\}$. A weight $\lambda\in X(T_n)$ is a weight of $M$ if and only if $M_\lambda \neq 0$. We have a decomposition $M=\bigoplus_{\lambda\in X(T_n)} M_\lambda$. The poset of dominant weights $X(T_n)^+\subset X(T_n)$ is the set $X(T_n)^+=\{\lambda:\langle \lambda,\alpha \rangle \geq 0 \text{ for all } \alpha\in \Phi^+\}=\{\lambda=(\lambda_1,\ldots,\lambda_n): \lambda_1\geq \ldots \geq	\lambda_n\}$ with the partial order defined as follows: $\lambda\geq \mu$ if and only if $\lambda-\mu=\sum_{1\leq i<j \leq n}c_{ij}(e_i-e_j)$ with $c_{ij}\leq 0$. Here, $\Phi^+=\{e_i-e_j:1\leq i<j \leq n\}\subset X(T_n)$, where $e_i$ is the $i$-th element of the standard basis of $X(T_n)$, is the system of positive roots corresponding to the Borel subgroup $B_n$. Let us note that the partial order given above is the reversed standard partial order on the set of dominant weights. A rational $\operatorname{GL}_n$-module $M$ is polynomial if and only if for all the weights $\lambda=(\lambda_1,\ldots,\lambda_n)$ of $M$ each $\lambda_i$ satisfies $\lambda_i\geq 0$ (c.f. \cite[Prop. 3.5]{fs}). The category $\operatorname{GL}^{\text{rat}}_n$-mod with the poset $X(T_n)^+$ is a highest weight category (c.f. \cite[Example 6.2]{ps}). 

Now we observe that
$\Lambda(d,n)$ is an order ideal of $X(T_n)^+$, hence by \cite[Thm. 3.9(a)]{cps_hwc}] the inclusion functor
\[ ev: {\cal P}_{d,n}\ra \operatorname{GL}^{\text{rat}}_n-\operatorname{mod}
\]
given explicitly as: $ev(F):=F({\bf k}^n)$,
and so its derived functor are full embeddings.
In particular, we have the isomorphisms of Ext-groups:
\[\operatorname{Ext}^*_{\mathcal{P}_{d,n}}(F,G) \simeq \operatorname{Ext}^*_{\operatorname{GL}^{\text{rat}}_n-\text{mod}}(F({\bf k}^n),G({\bf k}^n))
\] for any $F,G\in \mathcal{P}_{d,n}$
(see also  \cite[Cor. 3.12.1]{fs}, where it was re-proved without referring to \cite{cps_hwc}).

 Finally, let us recall that any finite dimensional rational $\operatorname{GL}_n$-module $M$ is of the form $M\simeq F({\bf k}^n)\otimes \det^r$ for some $F\in \mathcal{P}_{d,n}$ and a non-positive integer $r$, where $\det^r=(\Lambda^n({\bf k}^n))^{\otimes r}$ if $r>0$, $\det^0=id$ and $\det^r=((\Lambda^n({\bf k}^n))^{\otimes r})^*$ otherwise (c.f.\cite[Prop. 2.2.1]{weyman}).

Now we describe the categories $\mathcal{P}_{d,n}$ and $\mathcal{D}^b\mathcal{P}_{d,n}$ in terms of recollements of abelian and triangulated categories.
\begin{theorem}
\label{hwc_struct_bounded_thm}
Denote the full subcategory of \(\mathcal{P}_d \) consisting of strict polynomial functors of degree $d$, which assign the zero space to linear spaces with dimension at most $n$ by $\mathcal{P}^{> n}_d$. Let $i_*:\mathcal{P}^{> n}_d \to \mathcal{P}_d$ be the inclusion functor and let $j^*:\mathcal{P}_d\to \mathcal{P}_{d,n}$ be the exact functor, which restricts a given strict polynomial functor of degree $d$ to the category \( \Gamma^d\mathcal{V}ect_{\bf k}^{\leq n} \).
\begin{enumerate}
\item There is a recollement of abelian categories
\begin{equation*}
\begin{tikzcd}[row sep=huge, column sep=huge]
\mathcal{P}^{> n}_d \ar[r, "i_*"] & \mathcal{P}_d \ar[l, swap, bend left=30, "i^!"] \ar[l, swap, bend right=30, "i^*"] \arrow{r}{j^*} &  \mathcal{P}_{d,n} \ar[l, swap, bend left=30, "j_*"] \ar[l, swap, bend right=30, "j_!"]
\end{tikzcd}
\end{equation*}
and the recollement of triangulated categories induced by the recollement given above on the level of bounded derived categories
\begin{equation*}
\begin{tikzcd}[row sep=huge, column sep=huge]
\mathcal{D}\mathcal{P}^{> n}_d \ar[r, "i_*"] & \mathcal{D}\mathcal{P}_d \ar[l, swap, bend left=30, "Ri^!"] \ar[l, swap, bend right=30, "Li^*"] \arrow{r}{j^*} & \mathcal{D}\mathcal{P}_{d,n} \ar[l, swap, bend left=30, "Rj_*"] \ar[l, swap, bend right=30, "Lj_!"].
\end{tikzcd}
\end{equation*}
\item 
Let \( \lambda \in \Lambda(d,n) \). The costandard (resp. standard, simple) object corresponding to \( \lambda \) in $\mathcal{P}_{d,n}$ is the functor \(j^*(S_\lambda)\) (resp. \(j^*(W_{\lambda}) \), \(j^*(F_{\lambda})\)). Thus, costandard and standard functors in \( \mathcal{P}_{d,n} \) will be also called, respectively, Schur and Weyl functors and will be denoted by \( S_\lambda \) and \(W_\lambda\). The simple functor in $\mathcal{P}_{d,n}$ corresponding to $\lambda\in \Lambda(d,n)$ will be denoted by $F_\lambda$.
\item Let \( \lambda \in \Lambda(d,n) \). Then 
$j_*(S_\lambda)\simeq  Rj_*(S_{\lambda})\simeq   S_\lambda$ and $j_!(W_\lambda)\simeq Lj_!(W_{\lambda})\simeq W_\lambda.$\\ The functor $j_*$
(resp. $j_!$) preserves injective envelopes (resp. projective covers).
\end{enumerate}
\end{theorem}
\begin{proof}
The first part of (1) follows from \cite[Example 2.13]{psaroudakis}. We observe that $\mathcal{P}_d^{>n}$ consists of strict polynomial functors of degree $d$ with composition factors $F_\lambda$, where $\lambda$ is a Young diagram of weight $d$ and with at least $n+1$ columns. Then the inclusion functor $\mathcal{D}(\mathcal{P}_d^{>n}) \to \mathcal{D}(\mathcal{P}_d)$ is a full embedding by \cite[Thm. 3.9(a)]{cps_hwc}]. The categories \( \mathcal{P}_d^{>n}, \mathcal{P}_d \) and \( \mathcal{P}_{d,n} \) have finite global dimensions as highest weight categories with finite posets. Thus, the second part of (1) follows from \cite[Thm. 7.2(iii)]{psaroudakis}. The statement (2) follows from (1) and \cite[Lem. 1.4(b)]{cps_duality}.

Now we turn to the proof of (3). Fix $\mu \in \Lambda(d,n).$ By the right adjointness of $Rj_*$ to $j^*$ and part (2) we have \[ \operatorname{Hom}_{\mathcal{DP}_d}(W_\lambda,Rj_*(S_\mu)[t])\simeq\operatorname{Ext}^t_{\mathcal{P}_{d,n}}(j^*W_\lambda,S_\mu)\simeq\operatorname{Ext}^t_{\mathcal{P}_{d,n}}(W_\lambda,S_\mu)=0 \]for any $t\neq 0$, $\lambda \in \Lambda(d,n).$ Therefore $Rj_*(S_{\lambda})\simeq j_*S_\mu$ and it  is a good object (i.e. it has a filtration with the quotients being   Schur functors). Let us recall that  in any highest weight category ${\cal C}$ with poset $I$, for any $x,y\in I$  
we have ${\rm Hom}_{{\cal C}}(W_{x},S_{y})=\delta_{xy}$. Therefore
for a good $X\in\mathcal{P}_d$  we have \( \ell(X)=\dim\operatorname{Hom}_{\mathcal{P}_d}(W,X),\) where $\ell(X)$ is the length of a corresponding good filtration of $X$ and $W=\bigoplus_{\lambda\in \Lambda(d,n)}W_\lambda$. By the adjointness  and part (2) again, we obtain
\begin{equation*}
\ell(j_*(S_\mu))=\dim\operatorname{Hom}_{\mathcal{P}_d}(W,j_*(S_\mu))=\dim\operatorname{Hom}_{\mathcal{P}_{d,n}}(j^*(W),S_\mu)=\dim\operatorname{Hom}_{\mathcal{P}_{d,n}}(W,S_\mu)=
\end{equation*}
\begin{equation*}
\dim\operatorname{Hom}_{\mathcal{P}_{d,n}}(W_{\mu},S_\mu)=
1,
\end{equation*}
In consequence, $j_*(S_\mu)\simeq S_\mu$. The functor $j_*$ preserves injectives, because it has an exact left adjoint. Let $G\in{\cal P}_{d,n}$ be the injective envelope of  $F\in{\cal P}_{d,n}$. Then, since $j_*$ preserves monomorphisms, we have an embedding $j_*(F)\subset j_*(G)$. Suppose that
$j_*(G)$ is not the injective envelope of $j_*(F)$, i.e. there exists an injective $H\in{\cal P}_d$
such that  $j_*(F)\subset H\subset j_*(G)$. Since $H$ is injective, it is a direct summand of $j_*(G)$.
By applying $j^*$  we get
$F\subset j^*(H) \subset G$, since $j^* j_*\simeq id$. This gives a contradiction with the minimal property of envelope, because $j^*(H)$ is a direct summand of $G$, hence it is injective. This shows that $j_*$ preserves injective envelopes.

 The proofs of the respective facts for $j_!(W_\lambda)$ are analogous.
\end{proof}
We end the section by recalling the block structure of the Schur algebra \( S(n,d) \) (c.f. \cite{donkin}). By Theorem \ref{schur_equiv_prop} this is also the block structure of \( \mathcal{P}_{d,n} \). For \( \lambda\in \Lambda(d,n) \) with the conjugate Young diagram $\widetilde{\lambda}$ of the form \( \widetilde{\lambda}=(\lambda'_1,\ldots,\lambda'_n) \) we define \[ \alpha(\lambda)=\max\left\lbrace r \geq 0: \forall_{1\leq i \leq n-1} \ \lambda'_i-\lambda'_{i+1} \equiv -1 \ \left( \text{mod} \ p^r\right) \right\rbrace. \] Then, identifying a block with a subset of \( \Lambda(d,n) \), \(\lambda\) and \(\mu\) are in the same block of \( \mathcal{P}_{d,n}\) if and only if \( \lambda \) and \(\mu\) have the same \(p\)-core and \( \alpha(\lambda)=\alpha(\mu) \).

\section{The homological algebra in $\mathcal{P}_{p,n}$}
It is well-known that if $d<p$ then $S(n,d)$ is a semisimple algebra (c.f. \cite[Thm. 2.2.8]{martin}), hence $\mathcal{P}_{d,n}$ is a semisimple category for $d<p$. Thus, the first non-trivial case from the point of view of homological algebra is $d=p$. If $n\geq p$ then \( \mathcal{P}_{p,n} \simeq \mathcal{P}_p \) by Theorem \ref{schur_equiv_prop}. Therefore we assume in the sequel that $n<p$ and we call this situation {\em unstable}. Let us observe that $\alpha(\lambda)=0$ for any $\lambda\in \Lambda(p,n)$. It is also easily seen that each diagram in $\Lambda(p,n)$ not being a $p$-hook is a $p$-core and all $p$-hooks in $\Lambda(p,n)$ have the same $p$-core, namely the empty set. As the consequence, the only non-trivial homological computations appear in the block of $\mathcal{P}_{p,n}$ corresponding to the subset of $\Lambda(p,n)$ consisting of $p$-hooks, which will be denoted by $\mathcal{P}_{p,n}^{\varnothing}$. Set $S_i=S_{(i+1,1^{p-i-1})}, W_i=W_{(i+1,1^{p-i-1})}$ and $F_i=F_{(i+1,1^{p-i-1})}$ for $0\leq i \leq n-1$. 
 
We start with collecting the facts which are formal consequences of the recollement between ${\cal P}_p$ and ${\cal P}_{p,n}$  described in
Section 1 and the corresponding properties of ${\cal P}_p$ established in
\cite{pj}.
\begin{theorem}[]~
\label{homological_results}
\begin{enumerate}
\item
The decomposition matrix $D=(d_{ij})\in M_{n\times n}(\mathbb{Z})$ of $\mathcal{P}^{\varnothing}_{p,n}$ is given by \[d_{ij}=[W_{i-1}:F_{j-1}]=\begin{cases} 1 \qquad & if \ j=i \ or \ j=i+1 \\ 0 \qquad & otherwise \end{cases}.\]
\item Let $0\leq i,j \leq n-1.$ Then:
\begin{equation*}
\operatorname{Ext}^q(F_i,S_j)=\begin{cases} {\bf k} \quad if \ j\geq i \ and \ q=j-i \\ 0 \quad otherwise \end{cases}.
\end{equation*}
\item The category $\mathcal{P}_{p,n}$ has a Kazhdan-Lusztig theory relative to the function $l:\Lambda(p,n)\to \mathbb{Z}$ given by $l((i+1,1^{p-i-1}))=i$ for $0\leq i \leq n-1$ and $l(\lambda)=0$ for $\lambda$ not being a $p$-hook.
\item \begin{equation*}
\operatorname{Ext}^q(F_i,F_j)=\begin{cases} {\bf k} \quad if \ q=|i-j|+2r, \ where \ 0\leq r \leq n-\max\{i,j\}-1, \\ 0 \quad otherwise \end{cases}.
\end{equation*}
\item \begin{equation*}
\operatorname{Ext}^q(S_i,S_j)=\begin{cases} {\bf k} \quad if \ (i=j \land q=0) \lor (j>i \land (q=j-i-1 \lor q=j-i)) \\ 0 \quad otherwise \end{cases}.
\end{equation*}
\end{enumerate}
\end{theorem}
\begin{proof}
(1) follows from the fact that $j^*$ is exact and preserves $F_i$ and $S_i$ for
$0\leq i\leq n-1$.  The computation in (2) follows from \cite[Prop. 2.1]{pj}, since the groups under consideration are the same in ${\cal P}_{p,n}$ and ${\cal P}_{p}$ because of the equality $j_*(S_j)=S_j$. (3) follows from 
(2) (although it is also a formal consequence of the existence of the recollement setup). The computation of the Ext-groups in (4) is in turn a formal consequence of the existence of a Kazhdan-Lusztig theory in ${\cal P}_{p,n}$. The computation in (5) follows from  \cite[Prop. 2.6]{pj}
and the fact that $j_*(S_i)=S_i$. \end{proof}
Now we investigate multiplicative structures and formality phenomena in 
${\cal P}_{p,n}$. Analogous questions in ${\cal P}_p$ were studied by the second author in \cite{pj}. We will again use the close relation between
${\cal P}_p$ and ${\cal P}_{p,n}$. A general picture is quite similar in  both cases, but the crucial constructions of \cite{pj} need to be aptly adopted 
to ${\cal P}_{p,n}$. Namely, 
the second author constructed in \cite[Section 3.2]{pj}   an explicit injective resolution 
$\mathcal{R}_i$ of $F_i$
in ${\cal P}_p$ consisting of the parts of the de Rham and Koszul complexes spliced together. Unfortunately, $j^* (\mathcal{R}_i)$ will not be suitable for us, since it is not injective anymore. Nevertheless, we shall obtain an injective resolution $\mathcal{R}_{i,n}$ of $F_i$
in ${\cal P}_{p,n}$  by a similar construction.
Namely, set \[\mathcal{R}^{r,s}_{i,n} = \Omega^{i+s-r} \qquad if  \ 0\leq r \leq n-1, \ 0\leq s \leq n-i-1 \ and \ r-s \leq i ,\] 
where $\Omega^i:=S^{p-i}\otimes\Lambda^i$. 
The horizontal and vertical differentials are, respectively, the de Rham and Koszul differentials. Since $\kappa d+ d \kappa=0$, $\mathcal{R}_{i,n}^{**}$
is a double complex.
We observe that $\mathcal{R}_{i,n}$ is the truncation of the double complex providing $\mathcal{R}_i$ at the vertical degree $n-i-1$. For instance, $\mathcal{R}_{3,5}$ for $p=7$ is of the following form:
$$\begin{tikzcd}
\Omega^4 \arrow[r,"\kappa_4"] & \Omega^3 \arrow[r,"\kappa_3"] & \Omega^2 \arrow[r,"\kappa_2"] & \Omega^1 \arrow[r,"\kappa_1"] & \Omega^0\\
\Omega^3 \arrow[r,"\kappa_3"] \arrow[u,"d_3"] & \Omega^2 \arrow[r,"\kappa_2"] \arrow[u,"d_2"] & \Omega^1 \arrow[r,"\kappa_1"] \arrow[u,"d_1"] & \Omega^0 \arrow[u,"d_0"]
\end{tikzcd}.$$
Since $\Omega^k$ is injective for $0\leq k \leq n-1$ (as being a summand in the injective cogenerator  $S^{p,n}$)  and \( H^*(\operatorname{Tot}(\mathcal{R}_{i,n}))=H^0(\operatorname{Tot}(\mathcal{R}_{i,n}))=F_i$ (c.f \cite[Section 3.2]{pj}), $\operatorname{Tot}(\mathcal{R}_{i,n})$ is an injective resolution of $F_i$. We use for the simplicity the same symbol $\mathcal{R}_{i,n}$ for that. Set $\mathcal{R}=\bigoplus_{0\leq i \leq n-1} \mathcal{R}_{i,n}$ and $F=\bigoplus_{0\leq i \leq n-1} F_i$. 
We now observe that by truncation we get a quasi-isomorphism 
$j^*(\mathcal{R}_i)\simeq \mathcal{R}_{i,n}$, which, in particular, shows that
$j^*$ induces an epimorphism $\operatorname{Ext}^*_{\mathcal{P}_p}(F,F)\to\operatorname{Ext}^*_{\mathcal{P}_{p,n}}(F,F).$ Also, since the construction of $\mathcal{R}_{i,n}$ is analogous to that of $\mathcal{R}_i$, the arguments of \cite[Section 3.2]{pj} apply to the current situation and we obtain:

\begin{theorem}[]~
\label{yon_simple_thm}
\begin{enumerate}
\item Let $B=\bigoplus_{t\in \mathbb{N}} B_t$ be the graded algebra with grading
\begin{align*}
B_t= \operatorname{span}\{b^t_{ji}:0\leq i,j\leq n-1 \text{ and } t=|i-j|+2r, \text{ where } 0\leq r \leq n-\max\{i,j\}-1\} 
\end{align*}
for $0\leq t\leq 2n-2$ and $B_t=0$ for $t\geq 2n-1$, where $b^t_{ji}$ are the formal symbols for $i,j,t$ satisfying the above conditions. The multiplication on $B$ is given by $$b^t_{lm}\cdot b^u_{ji} = \begin{cases} b^{t+u}_{mi} \qquad & if \ j=l \ and \ u+t \leq 2n-i-m-2\\ 0 \qquad & otherwise \end{cases}.$$
Then there is a graded algebra isomorphism $
\operatorname{Ext}^*_{{\cal P}_{p,n}}(F,F)\simeq B.$
\item There is a graded algebra isomorphism $\operatorname{Ext}_{{\cal P}_{p,n}}^*(F_i,F_i)\simeq {\bf k}[x]/(x^{n-i})$ for $0\leq i \leq n-1$  and $x$ of degree $2$. In particular, $\operatorname{Ext}_{{\cal P}_{p,n}}^*(F_i,F_i)$ is a commutative algebra.
\item The algebra $\operatorname{End}^*(\mathcal{R})$ is a formal DG algebra, i.e. there exists a quasi-isomorphism of DG algebras $\eta:\operatorname{Ext}^*_{{\cal P}_{p,n}}(F,F)\to \operatorname{End}^*(\mathcal{R}),$ where $\operatorname{Ext}^*_{{\cal P}_{p,n}}(F,F)$ is regarded as the DG algebra with zero differential.
\item There is an equivalence of triangulated categories $\mathcal{D}^b\mathcal{P}_{p,n}^{\varnothing} \simeq \mathcal{D}^b(\operatorname{Ext}^*_{{\cal P}_{p,n}}(F,F)-\text{mod}^{\text{gr}})$.
\end{enumerate}
\end{theorem}
We remark  that similar results concerning the Yoneda algebra of Schur functors, analogous to \cite[Thm. 3.2, Cor. 3.3]{pj}, also hold, though we leave their formulation to the interested reader.

The computations in ${\cal P}_{p,n}$ obtained so far could be hardly  called surprising. However, when we try to compute the groups 
${\rm Ext}^*_{{\cal P}_{p,n}}(F_i, W_j)$ and ${\rm Ext}^*_{{\cal P}_{p,n}}(S_i, W_j)$, we encounter some interesting phenomena. Again, analogously 
to \cite{pj} we will use the minimal injective resolution of $W_i$, let us call it 
$L_{i,n}$ (in fact \cite[(11)]{pj} gives the minimal projective resolution of $S_i$, hence we consider the Kuhn dual case). Again
we  cannot  just apply $j^*$ to the dual of \cite[(11)]{pj}, since the resulting 
complex is not injective.
Instead, we need to adapt the idea of its construction to our context. Namely, this time we concatenate 
the truncated dual Koszul and  Koszul complexes, but we cut them off  earlier 
than it was done in \cite[(11)]{pj} in order to get an injective complex. Thus, we put as $L_{i,n}$ the following:
\[
(\Omega^{i+1})^{\#}\longrightarrow 
(\Omega^{i+2})^{\#}\longrightarrow\ldots\longrightarrow (\Omega^{n-1})^{\#}
\stackrel{\alpha}{\longrightarrow}\Omega^{n-1}
\longrightarrow\Omega^{n-2}\longrightarrow\ldots\longrightarrow\Omega^{0},
\]
where $\alpha=(\Omega^{n-1})^{\#}\stackrel{k^{\#}}{\longrightarrow}
(\Omega^{n})^{\#}\simeq\Omega^{n}
\stackrel{k}{\longrightarrow} \Omega^{n-1}$.
We remark that $L_{i,n}$ is the minimal injective resolution of $W_i$ in
${\cal P}_{p,n}$, since
$\Omega^q$ is the injectve envelope of $F_q$ for $0\leq q\leq n-1$   by \cite[Proposition 1.1]{pj} and \cite[Lemma 1.2]{cps_duality}.

Then by repeating the arguments from the proofs of \cite[Prop. 2.8, Cor. 2.9, Prop. 2.10]{pj} we obtain:
\begin{prop}
We have the following computations in ${\cal P}_{p,n}$:
\begin{enumerate}
\item For $0\leq i,j \leq n-1$ \[\operatorname{Ext}^q(S_i,F_j)=\begin{cases} {\bf k} \qquad & if \ (i<j \land q=j-i-1) \lor q=2n-i-j-2 \\ 0 \qquad & otherwise \end{cases} \] and \[\operatorname{Ext}^q(F_i,W_j)=\begin{cases} {\bf k} \qquad & if \ (i>j \land q=i-j-1) \lor q=2n-i-j-2 \\ 0 \qquad & otherwise \end{cases}.\]
\item For $0\leq i,j \leq n-1$ \[\operatorname{Ext}^q(S_i,W_j)=\begin{cases} {\bf k}  & if \ (i=j \land q=0)\lor q=2n-i-j-3 \lor q=2n-i-j-2  \\ 0  & otherwise \end{cases}.\]
\end{enumerate}
\end{prop}
We point out here that although our construction is similar to that 
in ${\cal P}_p$, we cannot say that the result is just the truncation of the corresponding computation in ${\cal P}_p$, since the Ext-groups are shifted. In particular, in this case $j^*$ is not epimorphic on the Ext-groups. Some even more striking examples of non-obvious behavior of $j^*$ we will provide in the next paragraph.

It was observed in \cite[p. 183-184]{ab}) that $I^{\otimes d}$  is not projective if $d>n$. The conceptual reason for this is that it is not a direct summand in a projective generator of ${\cal P}_{d,n}$ anymore. Hence, let us compute
some Ext-groups in ${\cal P}_{p,n}$ involving
$I^{\otimes p}$.
To this end, we shall use the adjunction $\{j^*,Rj_*\}$ between the categories
${\cal DP}_d$ and ${\cal DP}_{d,n}$. Let us start with an easy general observation:
\begin{prop}
For any $F\in{\cal P}_{d,n}$ and projective $P\in{\cal P}_d$ we have:
\[ {\rm Ext}^{q}_{{\cal P}_{d,n}}(j^*(P),F)\simeq 
H^q({\rm Hom}_{{\cal P}_{d}}(P,Rj_*(F))).
\]
\end{prop}
\begin{proof} It follows from the afore-mentioned adjunction and the fact 
that the functor ${\rm Hom}_{{\cal P}_{d}}(P,-)$ is exact, hence it commutes with cohomology. \end{proof}
Thus, we see that the task of computing ${\rm Ext}^*_{{\cal P}_{p,n}}(I^{\otimes p},F)$ is essentially reduced to that of computing $H^*(Rj_*(F))$.
For this we shall use $j_*$-acyclic resolutions in ${\cal P}_{p,n}$, which are simpler than injective ones. One of these is $K_{i,n}$, the truncated complex of ``Koszul kernels''
in ${\cal P}_{p,n}$ given by:
\[
S_i\stackrel{d}{\longrightarrow}S_{i+1}\stackrel{d}{\longrightarrow}\ldots
\stackrel{d}{\longrightarrow}S_{n-1},
\]
which  is a  resolution of $F_i$ 
(cf. \cite[(9)]{pj}). The second is the truncated de Rham complex $M_{i,n}$:
\[
\Omega^{i+1}\stackrel{d}{\longrightarrow}\Omega^{i+2}\stackrel{d}{\longrightarrow}\ldots
\stackrel{d}{\longrightarrow}\Omega^n,
\]
which is a resolution of $W_i$. In fact $M_{i,n}$  is the first part of the complex $L_{i,n}$, because $(\Omega^j)^\#\simeq \Omega^j$ for $j>0$. It is not injective, since $\Omega^n$ is not injective, but it is, as we will see,
$j_*$-acyclic. We shall also need the  complex $M_{i,n}'$  in ${\cal P}_p$
 given by:
 \[
\Omega^{i+1}\stackrel{d}{\longrightarrow}\Omega^{i+2}\stackrel{d}{\longrightarrow}\ldots
\stackrel{d}{\longrightarrow}\Omega^{n-1}\stackrel{\beta}{\longrightarrow}
S_{n-1},
\]
where the map $\beta$ is the composition of the canonical projection and the de Rham differential:
\[\Omega^{n-1}\longrightarrow 
S_{n-2}\stackrel{d}{\longrightarrow} S_{n-1}.\]
In the next theorem we still use the convention introduced in Theorem 1.4 that for ``generally known functors'' we denote their restriction $j^*$ by the same symbol. Thus, formula like $j_*(F)\simeq F$ actually means that for 
some $F\in{\cal P}_p$ we have $j^*(j_*(F))\simeq F$.  

Now we have:
\begin{theorem}
For any $0\leq i\leq n-1$ we have:
\begin{enumerate}
\item $Rj_*(F_i)=j_*(K_{i,n})\simeq K_{i,n}$. Hence $j_*(F_{n-1})\simeq Rj_*(F_{n-1})\simeq S_{n-1}$ and for 
$0\leq i<n-1$, we have:
\[
H^q(Rj_*(F_i))\simeq
\begin{cases}
F_i & {\rm for\ } q=0\\
F_{n} & {\rm for\ } q=n-i-1\\
0 & {\rm otherwise.}
\end{cases}
\]
\item $Rj_*(W_i)=j_*(M_{i,n})\simeq M_{i,n}'$. Hence $j_*(W_{n-1})\simeq Rj_*(W_{n-1})\simeq S_{n-1}$ and for 
$0\leq i<n-1$, we have:
\[
H^q(Rj_*(W_i))=
\begin{cases}
W_i & {\rm for\ } q=0\\
F_{n-1} & {\rm for\ } q=n-i-2,\ q=n-i-1\\
0 & {\rm otherwise.}
\end{cases}
\]
\end{enumerate}
\end{theorem}
\begin{proof} We first observe that, by Theorem 1.4(3), the complex $K_{i,n}$
is $j_*$-acyclic, hence it may be used for computing $Rj_*(F_i)$.
Since $j_*(S_q)\simeq S_q$ for $q\leq n-1$ (Theorem 1.4(3) again), we see
that the complexes $j_*(K_{i,n})$ and $K_{i,n}$ consist of isomorphic objects.
Then by the faithfulness of $j_*$, $j_*(d)\neq 0$, but since
${\rm dim}({\rm Hom}_{{\cal P}_p}(S_q,S_{q+1}))=1$ by 
\cite[Proposition 2.6]{pj}, we conclude that $j_*(d)=d$ up to a non-zero scalar. This shows that $Rj_*(K_{i,n})\simeq K_{i,n}$ in ${\cal DP}_p$. The formula for cohomology of $Rj_*(F_i)$ follows from the fact that the kernel (resp. cokernel)
of the map $d: S_q \longrightarrow S_{q+1}$ in ${\cal P}_p$ for $q< p$
is isomorphic to $F_q$ (resp. $F_{q+1}$).

For the second part, we start by invoking Theorem 1.4(3) again to show that 
$M_{i,n}$ is $j_*$-acyclic, since all its terms except $\Omega^n$ are injective while $\Omega^n\simeq S_{n-1}$ in ${\cal P}_{p,n}$. More precisely, by the fact mentioned before Proposition 2.3  that for $q<n$, 
$\Omega^q$ is the
injective envelope of $S_j$ and Theorem 1.4(3), we obtain that 
$j_*(\Omega^q)\simeq \Omega^q$ for $0\leq q<n$ and $j_*(\Omega^n)\simeq
S_{n-1}$. Then we conclude that $j_*(M_{i,n})\simeq M_{i,n}'$ by the similar argument to that used in the proof of the first part, since the Hom-spaces between the functors under inspection are one-dimensional by e.g. an elementary computation using the sum-diagonal adjunction. Then the fact 
that $H^0(M_{i,n}')\simeq W_i$ is the main point of \cite[(11)]{pj}. The rest of the formula for $H^*(Rj_*(W_i))$ could be derived  from the description of the kernel/cokernel of $\beta$, but it also follows immediately from the first part of the Theorem and the long exact sequence of cohomology applied to 
the distinguished triangle
\[ 
Rj_*(F_{i+1})\ra Rj_*(W_i)\ra Rj_*(F_i)
\]
for $i<n-1$ coming from the short exact sequence in ${\cal P}_{p,n}$:
\[
0\ra F_{i+1}\ra W_i\ra F_i\ra 0,
\]
 which is a consequence of Theorem 2.1(1).
\end{proof}
Now we can apply our computation and Proposition 2.4 to $P=I^{\otimes p}$
and compute the corresponding Ext-groups. In order to  describe an extra structure  carried by these Exts, which comes from the  action of the symmetric group $\Sigma_p$ on $I^{\otimes p}$,  we shall formulate our results in terms of ``the derived Schur functor'' $Rs$. Namely, classically (see e.g. \cite[Section 4]{martin}) one considers ``the Schur functor'' $s: {\cal P}_d\longrightarrow {\bf k}[\Sigma]_d$-mod given by
the formula:
\[
s(F):={\rm Hom}_{{\cal P}_d}(I^{\otimes d},F).
\]
In fact, over a field of characteristic zero $s$ establishes an equivalence
between ${\cal P}_d$ and the category $\ka[\Sigma_d]{\rm-mod}$ of finite dimensional $\ka$-representations of $\Sigma_d$. Over a field of positive characteristic $s$ is not an equivalence anymore, but it still preserves an important information.
For example, for a $p$-regular $\lambda$, $s$ takes the simple functor $F_{\lambda}$ to  the simple ${\bf k}[\Sigma_d]$-module
associated to $\lambda$, which we shall denote by $G_{\lambda}$ (in particular, we denote by $G_i$ the simple ${\bf k}[\Sigma_p]$-module associated to the corresponding hook diagram). Also, $s(S_{\lambda})$ and 
$s(W_{\lambda})$ can be explicitly described. They are called, respectively, the Specht and dual Specht modules and denoted by $Sp_{\lambda}$ and 
 $Sp'_{\lambda}$ (again, we will also use the notations: 
 $Sp_i$ and $Sp'_i$).
 
 Now, since in our situation $I^{\otimes p}$ is not projective, it is natural to consider its derived functor operating on the bounded derived categories:
 \[
 Rs: {\cal DP}_{p,n}\longrightarrow {\cal D}{\bf k}[\Sigma_p]{\rm -mod},
 \]
 given by the formula:
 \[
 Rs(F):={\rm RHom}_{{\cal P}_{p,n}}(I^{\otimes d},F),
 \]
 which translates to Ext-computations via the formula:
 \[
 H^t(Rs(F))\simeq {\rm Ext}_{{\cal P}_{p,n}}^t(I^{\ot p}, F).
 \]
 Thus, within the formalism of total derived functors,
we can formulate our Ext-computations as follows (for completeness, we also add the obvious computation of $Rs(S_{\lambda}$):
\begin{cor}
There are the following isomorphisms in 
${\cal D}{\bf k}[\Sigma]_p{\rm -mod}$:
\begin{enumerate}
\item For any $0\leq i\leq n-1$, $Rs(S_i)=Sp_i$.
\item $Rs(F_0)=G_n[-(n-1)]$, $Rs(F_{n-1})=Sp_{n-1}$,
and for $0<i<  n-1$:
\[
 {\rm Ext}_{{\cal P}_{p,n}}^q(I^{\ot p}, F_i)\simeq H^q(Rs(F_i))=
\begin{cases}
G_i & {\rm for\ } q=0\\
G_{n} & {\rm for\ } q=n-1-i\\
0 & {\rm otherwise.}
\end{cases}
\]
\item $Rs(W_{n-1})=Sp_{n-1}$,
and for $0\leq i<  n-1$:
\[
 {\rm Ext}_{{\cal P}_{p,n}}^q(I^{\ot p}, W_i)\simeq H^q(Rs(W_i))=
\begin{cases}
Sp'_i & {\rm for\ } q=0\\
G_{n} & {\rm for\ } q=n-i-2,\  q=n-i-1\\
0 & {\rm otherwise.}
\end{cases}
\]
\end{enumerate}
\end{cor}
\begin{proof} Straightforward. \end{proof}
{\bf Remark:} 
Let us point out for some interesting phenomena here. Firstly, our formulas are overlapping at the end of the poset, since in ${\cal P}_{p,n}$ we
have $S_{n-1}\simeq F_{n-1}\simeq W_{n-1}$ (because $(n,1^{p-n})$ is minimal in its block). This, for example, shows that $j^*$ is not full, since ${\rm Hom}_{{\cal P}_{p}}(I^{\otimes p}, F_{n-1})=G_{n-1}$ while ${\rm Hom}_{{\cal P}_{p,n}}(I^{\otimes p}, F_{n-1})=Sp_{n-1}$. At the other end of the scale something interesting also happens. Namely,
since $(1^p)$ is $p$-singular we have $s(F_0)=0$, but, as we see,
$Rs(F_0)$ is not only nonzero but it is the shifted $G_{n}$, which could not be hit 
classically, since we do not have $F_n$ in ${\cal P}_{p,n}$.  

These phenomena have  quite a surprising consequence for $n=p-1$.
We recall that for an abelian category ${\cal A}$, $K_0({\cal D}{\cal A})$ denotes
the abelian group generated by objects  of the bounded derived category of ${\cal A}$ with the relations coming from distinguished triangles (see e.g. \cite[Section 2]{cps_kl}). Since we have the
relation $X[1]=-X$, $K_0({\cal D}{\cal A})$ is isomorphic to the usual $K_0({\cal A})$, hence, for an
artinian abelian category ${\cal A}$, it
is a free abelian group of rank equal to the number of non-isomorphic simples in ${\cal A}$. Now we have:
\begin{cor}
$Rs$ induces an isomorphism: 
\[
K_0(Rs): K_0({\cal DP}_{p,p-1})\simeq 
 K_0({\cal D}{\bf k}[\Sigma_p]{\rm -mod}).
 \]
 \end{cor}
\begin{proof}  It suffices to show that
$K_0(Rs)$ is onto. To this end, we first observe that $\pm[G_{p-1}]$ lies in the image, since $Rs(F_0)=G_{p-1}[-(p-2)]$.
Then, by the rest 
of computations in Corollary 2.6, we see that all $[G_{i}]$
for $0\leq i<p-1$
 lie in the image of $Rs$ too. \end{proof} 
This  can be contrasted with the fact that $Rs$ is extremely far from being an isomorphism. It is not only because ${\bf k}[\Sigma_p]{\rm -mod}$ has infinite homological dimension, but often $s$ acts trivially on Ext-groups
just for a dimension reason.
For example, we have 
${\rm Ext}^*_{{\cal P}_{p,p-1}}(F_0,F_0)={\bf k}[x]/x^{p-1}$ with $|x|=2$ while 
${\rm Ext}^*_{{\bf k}[\Sigma_p]{\rm -mod}}(G_{p-1},G_{p-1})\simeq
H^*(\Sigma_p,{\bf k})$, which is a subring of 
${\bf k}[x]\otimes\Lambda^*(y)$ with $|x|=2, |y|=1$ consisting of elements
of degree congruent to $0$ or $-1$ modulo $2p$.

\section{Spanier-Whitehead duality}
In this section we introduce a duality relating certain subcategories of 
${\cal P}_{d,n}$'s.
Let $\Lambda(d,n,k)$ be the set of Young diagrams of weight $d$
consisting of at most $n$ columns and at most $k$ rows.
We consider the category  ${\cal P}_{d,n,k}$, which is the full subcategory of
${\cal P}_{d,n}$ generated by the simples $F_{\la}$ for $\la$ in $\Lambda(d,n,k)$.  Since $\Lambda(d,n,k)$  is an order ideal in $\Lambda(d,n)$, we are again in the framework of \cite{cps_hwc}. In particular, ${\cal P}_{d,n,k}$ inherits from 
${\cal P}_{d,n}$ the structure of a highest weight category and the embedding
$i_*': {\cal P}_{d,n,k}\ra {\cal P}_{d,n}$ extends to a recollement of abelian categories. Let $ev':=ev\circ i_*'$.
Then the essential image of  $ev'$ can be characterized as consisting of representations with all the weights having all entries non-negative but not exceeding $k$.

Let us consider the bijection between $\Lambda(d,n,k)$ and 
$\Lambda(nk-d,n,k)$, which
sends a Young diagram $\la$ to its complement in the rectangle $n\times k$
turned upside down, i.e. $\la=(\la_1,\ldots,\la_k)\mapsto\widehat{\la}:=
(n-\la_k,\ldots,n-\la_1)$. Our goal is to extend this bijection to an equivalence of highest weight categories, which we shall call the Spanier-Whitehead duality, $\Phi_{d,n,k}: {\cal P}_{d,n,k}\simeq {\cal P}_{nk-d,n,k}$.
\begin{theorem}
There exists a functor $\Phi_{d,n,k}: {\cal P}_{d,n,k}\ra {\cal P}_{nk-d,n,k}$ satisfying the following properties:
\begin{enumerate}
\item 
Let $t_{n,k}: \operatorname{GL}_n^{rat} {\rm -mod}\ra \operatorname{GL}_n^{rat} {\rm -mod}$ be the functor given by $M\mapsto M\ot \det^k$. Then there is an isomorphism of contravariant functors: $ev'\circ \Phi_{d,n,k}\circ (-)^{\#}\simeq
t_{n,k}\circ (-)^{*}\circ ev'$.
\item $\Phi_{d,n,k}$ is an equivalence of $\ka$-linear abelian categories.
\item We have: 
$\Phi_{d,n,k}(S_{\la})\simeq 
S_{\hat{\la}}$, 
$\Phi_{d,n,k}(W_{\la})\simeq W_{\hat{\la}}$, 
$\Phi_{d,n,k}(F_{\la})\simeq 
F_{\hat{\la}}$, i.e. $\Phi_{d,n,k}$ is an equivalence of highest weight categories.
\end{enumerate}
\end{theorem}
\begin{proof} We define $\Phi_{d,n,k}$ by the formula:
\[ \Phi_{d,n,k}(F)(V):={\rm Hom}_{{\cal P}_{nk,n}}(W_P,i_*'(F)\ot S^{nk-d}_V)\]
for the Young diagram $P=(n^k)$. It is clear from this definition that  $\Phi_{d,n,k}$ targets ${\cal P}_{nk-d,n,k}$, however, it is a priori not obvious that its image is in fact contained in ${\cal P}_{nk-d,n,k}$. This will be showed in the course of the proof. We start  with establishing the first assertion of Theorem. Since the weights of $M^*$ are the negatives of those of $M$, we see that  for $F\in{\cal P}_{d,n,k}$ the weights of $F(\ka^n)^*\ot \det^k$ have all entries non-negative, hence it belongs to the image of $ev$. Then we consider the functor $ev^!$ right adjoint to
$ev$. Its existence follows from the Special Adjoint Functor Theorem, but,
by the Yoneda lemma, it can also be explicitly defined as:
\[
ev^!(M)(V):={\rm Hom}_{\operatorname{GL}_n^{rat}{\rm-mod}}(ev(\Gamma^d_{V^*}),M).
\]
Since $ev$ is a full embedding, we
have: $ev^!\circ ev\simeq id$, hence 
\[F(\ka^n)^*\ot{\rm det}^k\simeq
ev\circ ev^{!}(F(\ka^n)^*\ot {\rm det}^k).\]
This shows that in order to obtain the first part of the Theorem it suffices to show that  
\[
ev^{!}(F(\ka^n)^*\ot {\rm det}^k)\simeq \Phi_{d,n,k}
(F^{\#}).\]
 Applying the formula for $ev^!$ and using the tensor-hom adjunction in $\operatorname{GL}_n^{rat}{\rm -mod}$ we obtain:
\[
ev^{!}(F(\ka^n)^*\ot {\rm det}^k)(V)=
{\rm Hom}_{\operatorname{GL}_n^{rat}{\rm -mod}}(\Gamma^{nk-d}_{V^*}(\ka^n),F(\ka^n)^*\ot
{\rm det}^k)\simeq
\]
\[{\rm Hom}_{\operatorname{GL}_n^{rat}{\rm -mod}}(\Gamma^{nk-d}_{V^*}(\ka^n)\ot F(\ka^n),
{\rm det}^k).
\]
Now, since $\det^k=ev(S_P)$ and $ev$ is faithfully full, we get:
\[
ev^!(F(\ka^n)^*\ot {\rm det}^k)=
{\rm Hom}_{{\cal P}_{nk,n}}(\Gamma^{nk-d}_{V^*}\ot i'_*(F),
S_P)\simeq {\rm Hom}_{{\cal P}_{nk,n}}(W_P,S^{nk-d}_{V}\ot i'_*(F^{\#})),
\]
which gives the first part of Theorem. 

 At this point we can at last see that the image of 
$\Phi_{d,n,k}$ lies in ${\cal P}_{nk-d,n,k}$, since by the first part of 
Theorem 3.1 it consists of 
functors having (after taking $ev$) all the weights with the entries not exceeding $k$.

 We turn to the second part. 
Obviously, it suffices to show that $\Phi_{d,n,k}\circ (-)^{\#}$ is an equivalence.
By using the first part again and the fact that $\det^k\ot(\det^k)^*\simeq \ka$, we obtain
\[
ev'\circ \Phi_{d,n,k}\circ (-)^{\#} \circ \Phi_{d,n,k}\circ (-)^{\#}\simeq
t_{n,k}\circ (-)^{*}\circ t_{n,k}\circ (-)^{*}\circ ev'\simeq ev',
\]
which shows that $\Phi_{-,n,k}\circ (-)^{\#}$ is an involution. \newline
Now we turn to the third part. The crucial is the following combinatorial lemma.
\begin{lem}
For any $\la\in \Lambda(d,n,k)$ there is an isomorphism in ${\cal P}_{nk-d,n}$:
$S_{\widehat{\la}}\simeq S_{P/\la}$.
\end{lem}
\begin{proof} By the Littlewood-Richardson Formula, there is a filtration of $S_{P/\la}$ with the graded object 
consisting of Schur functors $S_{\mu}$. 
Our task is to show that among  $S_{\mu}$'s
the only one with $\mu_1\leq n$ is $\widehat{\la}$ and it has  multiplicity $1$. Combinatorially, it means that there is the exactly one Yamanouchi word of shape $P/\la$ and content $\mu$ with $\mu_1\leq n$, and its content is $\hat{\la}$. This is an elementary exercise: we see that such a word must have $1$'s in the places $(i,j)$ for $\la_i< j \leq \la_{i-1}$ by the Yamanouchi condition. Thus, we see that the number of $1$'s is equal to $n-\la_k$.
Then by the same argument the positions of $2$'s are uniquely determined. By continuing in this manner we construct the single word, which is of content $\widehat{\la}$. \end{proof}

Then by using the relationship between ${\cal P}_{nk,n}$ and ${\cal P}_{nk}$, the sum-diagonal adjunction and the Decomposition Formula, we obtain:
\[\Phi_{d,n,k}(S_{\la})(V)={\rm Hom}_{{\cal P}_{nk,n}}\left(W_P,S_{\la}\ot S^{nk-d}_V\right)\simeq
{\rm Hom}_{{\cal P}_{nk}}\left(W_P,S_{\la}\ot S^{nk-d}_V\right)\simeq
\]
\[{\rm Hom}_{{\cal P}_{d}\times{\cal P}_{nk-d}}
\left(W_P(-\oplus -),S_{\la}\hat{\ot} S^{nk-d}_V\right)\simeq 
{\rm Hom}_{{\cal P}_{d}\times{\cal P}_{nk-d}}
\left(\sum_{\mu\in \Lambda(d,n,k)}
W_{\mu}\hat{\ot} W_{P/\mu},S_{\la}\hat{\ot} S^{nk-d}_V\right)
\]
\[
\simeq 
{\rm Hom}_{{\cal P}_{nk-d}}
\left(W_{P/\la},S^{nk-d}_V\right)\simeq S_{P/\la}(V)\simeq S_{\widehat{\la}}(V),\]
 which gives the first isomorphism of the third part of the theorem.
 In order to obtain the second isomorphism 
 we recall that co-good objects in a highest weight category are characterized by the condition that they have trivial higher Ext's with costandard objects. This shows that $\Phi_{d,n,k}(W_{\la})$ is co-good, because 
 $\Phi_{d,n,k}$ induces 
isomorphisms on Ext's. Now we conclude that  $\Phi_{d,n,k}(W_{\la})=W_{\widehat{\la}}$, 
because ${\rm Hom}_{{\cal P}_{nk-d,n,k}}(\Phi_{d,n,k}
 (W_{\la}),S_{\mu})=\delta_{\la\widehat{\mu}}$.
The third isomorphism follows from the fact that a simple object in a highest weight category is the socle of the respective costandard object. \end{proof}

Let us discuss some corollaries to Theorem 3.1 and its relation to the duality in $\operatorname{GL}_n^{rat}$-mod. 
\begin{cor}
Let $\la, \mu\in \Lambda(d,n,k)$, $s\geq 0$ and $G_{\la}, G_{\mu}$ be the Schur, Weyl or simple functors associated to $\la,\mu$ (the choice is the same in a given formula). Then:
\begin{enumerate}
\item There is an  isomorphism in $\operatorname{GL}_n^{rat}$-mod:
\[G_{\la}(\ka^n)^*\ot det^k\simeq G_{\hat{\la}}(\ka^n).
\]
\item There is an isomorphism of graded vector spaces:
\[
{\rm Ext}^*_{{\cal P}_{d,n,k}}(G_{\la}, G_{\mu})\simeq
{\rm Ext}^*_{{\cal P}_{nk-d,n,k}}(G_{\widehat{\la}}, G_{\widehat{\mu}}).
\]
Analogous isomorphisms of Ext-groups hold in the categories:
${\cal P}_{d,n}$ vs. ${\cal P}_{nk-d,n}$, $\operatorname{GL}_n^{rat}$-mod and also, with the exception of the case of simple functors, in  ${\cal P}_d$ vs. 
${\cal P}_{nk-d}$.
\end{enumerate}
\end{cor}
\begin{proof} The first statement follows immediately from the parts 1 and 3 of Theorem 3.1. The second statement for ${\cal P}_{d,n,k}$  follows from the parts 2 and 3 of Theorem 3.1. The case of ${\cal P}_{d,n}$ and the case of $\operatorname{GL}_n^{rat}$-mod  are  consequences of the full embeddings of respective derived categories, which in turn follow from the 
existence of the recollement setups. The result is that the Ext-groups under inspection are just the same in all these categories.
The case of the category ${\cal P}_d$ is a bit more specific. The Ext-groups 
between Schur functors and between Weyl functors are still  the same 
in ${\cal P}_d$ and ${\cal P}_{d,n}$ because the left adjoint 
to the projection ${\cal P}_d\ra{\cal P}_{d,n}$ preserves the Weyl functors while the right adjoint preserves  the Schur functors. \end{proof}

The first part of Theorem 3.1 indicates to a very close connection between $\Phi_{d,n,k}$ and the monoidal duality in $\operatorname{GL}_n^{rat}$-mod. Also, our proof of the fact that $\Phi_{d,n,k}$ is an equivalence heavily uses the fact that $(-)^*$ is an equivalence. Thus, one should not expect that our Spanier-Whitehead duality sheds 
a new light on representations of $\operatorname{GL}_n$. For example, the parts of Corollary 3.3 concerning $\operatorname{GL}_n$ seem to be well-known to experts (see. e.g.  \cite[Exercise 2.18]{weyman}). That being said, one should not think that $\Phi_{d,n,k}$ is ``just twisted $(-)^*$
restricted to ${\cal P}_{d,n,k}$''. We would rather say that it  points out for a certain categorical phenomenon behind the isomorphisms of Corollary 3.3(1). To explain this point, let us observe that although  both sides of the isomorphism are expressed in terms of natural constructions, a priori there is no evident way of interpreting it functorially. The most obvious reason is that the left-hand side depends on $G$ contravariantly while the right-hand side does covariantly. But even disregarding this observation, we see that just  dimensions
of $G_{\la}(V)^*\ot\Lambda^n(V)^{\ot k}$ and $G_{\widehat{\la}}(V)$ differ drastically (unless $\dim(V)=n$). Thus, we found it quite surprising that the equivalence $(-)^*$, which is ``just taking the linear dual'', leads to the
equivalence $\Phi_{d,n,k}$ on the level of functor categories, which changes dimensions a lot. 
 
Let us illustrate our point by an example, which also shows that the exception in Corollary 3.3(2) is necessary. We consider the diagram $\la=(1^p)\in \Lambda(p,n,p)$ for
$n>1$. Then $\widehat{\la}=((n-1)^p)$ and the functors $S_{\la}=S^p$, $F_{\la}=I^{(1)}$ and
$S_{\widehat{\la}}=S_{((n-1)^p)}$, $F_{\widehat{\la}}=(\Lambda^{n-1})^{(1)}$ apparently have nothing in common.
However, by Theorem 3.1, e.g. $S_{\widehat{\la}}$ has the composition series of length 2 (more precisely, we have: $S_{((n-1)^p)}/(\Lambda^{n-1})^{(1)}=
F_{(n,(n-1)^{p-2},n-2)}$), which is a priori highly non-trivial. At last, when we take Ext-groups, we  have that:
\[
\operatorname{Ext}^{*}_{{\cal P}_{p,n}}(I^{(1)},I^{(1)})\simeq 
\operatorname{Ext}^{*}_{{\cal P}_{pn-p,n}}((\Lambda^{n-1})^{(1)},(\Lambda^{n-1})^{(1)}),
\]
which again is non-obvious, especially when we take into account the fact that the analogous Ext-groups in ${\cal P}_p$ and ${\cal P}_{np-p}$ are very different
by \cite[Thm. 5.6]{ffss}. The last observation also shows that the isomorphism
from Corollary 3.3(2) for the simple functors fails in ${\cal P}_p$ vs. ${\cal P}_{pn-p}$.

Yet another interesting phenomenon is that since ${\cal P}_{d,n,k}$ is semisimple for $d<p$, by Theorem 3.1(2) the same holds for 
${\cal P}_{nk-d,n,k}$. Thus, we obtain a certain semisimple subcategory in the category of functors of large degree.

In a similar spirit, if $n\geq d$ then we have ${\cal P}_{d,n}\simeq {\cal P}_d$,
hence in such a case ${\cal P}_{d,n,k}$ is just a subcategory of ${\cal P}_{d}$.
However, for  the corresponding category ${\cal P}_{nk-d,n,k}$ ``the stability condition'' $nk-d\geq n$ is never satisfied unless $k\leq 2$, where the duality fizzles out to an auto-equivalence isomorphic to the identity.
Therefore we see that our duality is truly an unstable phenomenon.

We finish the article by explaining the name of our duality. We first recall how the classical S-W duality works. We start with, say, the category of CW-complexes
with the monoidal structure given by the smash product. Then we obtain the stable category by formally inverting the operation of smashing with $S^1$.
In this framework,  the Spanier-Whitehead duality says that the monoidal dual in the stable category of a small enough space $X$ (i.e. a good subspace of $S^k$) has the explicit description: it is $S^k\setminus X$ smashed with $S^{-k}$.

Now one can see that our duality $\Phi_{d,n,k}$ formally resembles that topological construction. This time, we start with ${\cal P}_{-,n}:=\oplus_{d\geq 0}{\cal P}_{d,n}$ with the monoidal structure given by tensoring over $\ka$. Then $\operatorname{GL}_n^{rat}$-mod is obtained from ${\cal P}_{-,n}$ by formally inverting the operation of tensoring with $\Lambda^n$. Thus, Theorem 3.1(1) may be thought of as follows: for an object $F$ small enough (i.e. coming from ${\cal P}_{-,n,k}$) we express the monoidal dual of $F$ as applying a certain explicit construction in ${\cal P}_{-,n}$ 
(i.e. $\Phi_{-,n,k}\circ (-)^{\#}$)  to $F$ and then tensoring with 
$(\Lambda^n)^{-k}$. To make this analogy even more striking, let us  recall
that our basic combinatorial construction $\la\mapsto\widehat{\la}$ is just taking the complement of $\la$ in the rectangle $P$ which labels 
$S_P\simeq \det^k$, which is the analogue of $k$-dimensional sphere.

\textsc{Marcin Chałupnik, Institute of Mathematics, University of Warsaw, Banacha 2, 02-097 Warsaw, Poland}

\textit{E-mail address}: \texttt{mchal@mimuw.edu.pl}\\

\textsc{Patryk Jaśniewski, Institute of Mathematics, University of Warsaw, Banacha 2, 02-097 Warsaw, Poland}

\textit{E-mail address}: \texttt{p.jasniewski@mimuw.edu.pl}
\end{document}